\documentclass[12pt]{article}
\usepackage{amsfonts}
\usepackage{amsmath}
\usepackage{amssymb}
\usepackage{amsthm}
\newtheorem{theorem}{Theorem}[section]
\newtheorem{lemma}{Lemma}[section]

\theoremstyle{definition}
\newtheorem{definition}{Definition}[section]
\newtheorem{remark}{Remark}[section]

\begin{document}

\begin{center}
\vskip 1cm{\LARGE\bf On the Density of Ranges of Generalized Divisor Functions with Restricted Domains 
\vskip 1cm
\large
Colin Defant\footnote{This work was supported by National Science Foundation grant no. 1262930.}\\
Department of Mathematics\\
University of Florida\\
United States\\
cdefant@ufl.edu}
\end{center}
\vskip .2 in

\begin{abstract}
We begin by defining functions $\sigma_{t,k}$, which are generalized divisor functions with restricted domains. For each positive integer $k$, we show that, for $r>1$, the range of $\sigma_{-r,k}$ is a subset of the interval $\displaystyle{\left[1,\frac{\zeta(r)}{\zeta((k+1)r)}\right)}$. After some work, we define constants $\eta_k$ which satisfy the following: If $k\in\mathbb{N}$ and $r>1$, then the range of the function $\sigma_{-r,k}$ is dense in $\displaystyle{\left[1,\frac{\zeta(r)}{\zeta((k+1)r)}\right)}$ if and only if $r\leq\eta_k$. We end with an open problem.  
\end{abstract}
 
\section{Introduction} 
Throughout this paper, we will let $\mathbb{N}$ denote the set of positive integers, and we will let $\mathbb{P}$ denote the set of prime numbers. We will also let $p_i$ denote the $i^{th}$ prime number. 
\par 
For a real number $t$, define the function $\sigma_t\colon\mathbb{N}\rightarrow\mathbb{R}$ by $\displaystyle{\sigma_t(n)=\sum_{\substack{d\vert n \\ d>0}}d^t}$ 
for all $n\in\mathbb{N}$. Note that $\sigma_t$ is multiplicative for any real $t$. For each positive integer $n$, if $r>1$, we have $\displaystyle{1\leq\sigma_{-r}(n)=\sum_{\substack{d\vert n \\ d>0}}\frac{1}{d^r}<\sum_{i=1}^{\infty}\frac{1}{i^r}}=\zeta(r)$, where $\zeta$ denotes the Riemann zeta function. The author has shown \cite{Defant14} that if $r>1$, then the range of the function $\sigma_{-r}$ is dense in the interval $[1,\zeta(r))$ if and only if $r\leq\eta$, where $\eta$ is the unique number in the interval $(1,2]$ that satisfies the equation $\displaystyle{\left(\frac{2^{\eta}}{2^{\eta}-1}\right)\left(\frac{3^{\eta}+1}{3^{\eta}-1}\right)=\zeta(\eta)}$.  
\par 
For each positive integer $k$, let $S_k$ be the set of positive integers defined by 
\[S_k=\{n\in\mathbb{N}:p^{k+1}\nmid n\hspace{1 mm}\forall\hspace{1 mm} p\in\mathbb{P}\}.\]
For any real number $t$ and positive integer $k$, let $\sigma_{t,k}\colon S_k\rightarrow\mathbb{R}$ be the restriction of the function $\sigma_t$ to the set $S_k$, and let $\log\sigma_{t,k}=\log\circ\hspace{0.75 mm}\sigma_{t,k}$. We observe that, for any $k\in\mathbb{N}$ and $r>1$, the range of $\sigma_{-r,k}$ is a subset of $\displaystyle{\left[1,\frac{\zeta(r)}{\zeta((k+1)r)}\right)}$. This is because, 
if we allow $\displaystyle{\prod_{i=1}^v q_i^{\beta_i}}$ to be the canonical prime factorization of some positive integer in $S_k$ (meaning that $\beta_i\leq k$ for all $i\in\{1,2,\ldots,v\}$), then 
\[1=\sigma_{-r,k}(1)\leq\sigma_{-r,k}\left(\prod_{i=1}^v q_i^{\beta_i}\right)=\prod_{i=1}^v \sigma_{-r,k}(q_i^{\beta_i})=\prod_{i=1}^v \left(\sum_{j=0}^{\beta_i}q_i^{-jr}\right)\] 
\[\leq\prod_{i=1}^v \left(\sum_{j=0}^k q_i^{-jr}\right)<\prod_{i=1}^{\infty}\left(\sum_{j=0}^k p_i^{-jr}\right)=\prod_{i=1}^{\infty}\frac{1-p_i^{-(k+1)r}}{1-p_i^{-r}}=\frac{\zeta(r)}{\zeta((k+1)r)}.\] 
To simplify notation, we will write $\displaystyle{G_k(r)=\frac{\zeta(r)}{\zeta((k+1)r)}}$. 
\par 
Our goal is to analyze the ranges of the functions $\sigma_{-r,k}$ in order to find constants analogous to $\eta$ for each positive integer $k$. More formally, for each $k\in\mathbb{N}$, we will find a constant $\eta_k$ such that if $r>1$, then the range of $\sigma_{-r,k}$ is dense in $[1,G_k(r))$ if and only if $r\leq\eta_k$. 
 
\section{The Ranges of $\sigma_{-r,k}$}  
\begin{definition} \label{Def2.1} 
For $k,m\in\mathbb{N}$ and $r\in(1,\infty)$, let 
\[f_k(m,r)=\log\left(1+\frac{1}{p_m^r}\right)+\sum_{i=1}^m\log\left(\sum_{j=0}^k\frac{1}{p_i^{jr}}\right).\]
\end{definition}
Notice that, for any $k\in\mathbb{N}$ and $r\in(1,\infty)$, the range of $\sigma_{-r,k}$ is dense in the interval $[1,G_k(r))$ if and only if the range of $\log\sigma_{-r,k}$ is dense in the interval $[0,\log (G_k(r)))$. For this reason, we will henceforth focus on the ranges of the functions $\log\sigma_{-r,k}$ for various values of $k$ and $r$. 
\begin{theorem} \label{Thm2.1} 
Let $k\in\mathbb{N}$, and let $r\in(1,\infty)$. The range of $\log\sigma_{-r,k}$ is dense in the interval $[0,\log(G_k(r)))$ if and only if $f_k(m,r)\leq\log(G_k(r))$ for all $m\in\mathbb{N}$. 
\end{theorem} 
\begin{proof} 
First, suppose that there exists some $m\in\mathbb{N}$ such that $f_k(m,r)>\log(G_k(r))$. Then 
\[\log\left(1+\frac{1}{p_m^r}\right)+\sum_{i=1}^m\log\left(\sum_{j=0}^k\frac{1}{p_i^{jr}}\right)>\log\left(\prod_{i=1}^{\infty}\left(\sum_{j=0}^k\frac{1}{p_i^{jr}}\right)\right)\] 
\[=\sum_{i=1}^{\infty}\log\left(\sum_{j=0}^k\frac{1}{p_i^{jr}}\right),\] 
which means that 
\[\log\left(1+\frac{1}{p_m^r}\right)>\sum_{i=m+1}^{\infty}\log\left(\sum_{j=0}^k\frac{1}{p_i^{jr}}\right).\]
Fix some $N\in S_k$, and let $\displaystyle{N=\prod_{i=1}^v q_i^{\gamma_i}}$ be the canonical prime factorization of $N$. Note that $\gamma_i\leq k$ for all $i\in\{1,2,\ldots,v\}$ because $N\in S_k$. If $p_s\vert N$ for some $s\in\{1,2,\ldots,m\}$, then 
\[\log\sigma_{-r,k}(N)\geq\log\left(1+\frac{1}{p_s^r}\right)\geq\log\left(1+\frac{1}{p_m^r}\right).\]
On the other hand, if $p_s\nmid N$ for all $s\in\{1,2,\ldots,m\}$, then 
\[\log\sigma_{-r,k}(N)=\log\left(\prod_{i=1}^v \sigma_{-r,k}(q_i^{\gamma_i})\right)=\log\left(\prod_{i=1}^v \left(\sum_{j=0}^{\gamma_i}\frac{1}{q_i^{jr}}\right)\right)\] 
\[\leq\log\left(\prod_{i=1}^v \left(\sum_{j=0}^k\frac{1}{q_i^{jr}}\right)\right)<\log\left(\prod_{i=m+1}^{\infty}\left(\sum_{j=0}^k\frac{1}{p_i^{jr}}\right)\right)\] 
\[=\sum_{i=m+1}^{\infty}\log\left(\sum_{j=0}^k\frac{1}{p_i^{jr}}\right).\] 
Because $N$ was arbitrary, this shows that there is no element of the range of $\log\sigma_{-r,k}$ in the interval $\displaystyle{\left(\sum_{i=m+1}^{\infty}\log\left(\sum_{j=0}^k\frac{1}{p_i^{jr}}\right),\log\left(1+\frac{1}{p_m^r}\right)\right)}$. Therefore, the range of $\log\sigma_{-r,k}$ is not dense in $[0,\log(G_k(r)))$. 
\par  
Conversely, suppose that $f_k(m,r)\leq\log(G_k(r))$ for all $m\in\mathbb{N}$. This is equivalent to the statement that 
\[\log\left(1+\frac{1}{p_m^r}\right)\leq\sum_{i=m+1}^{\infty}\log\left(\sum_{j=0}^k\frac{1}{p_i^{jr}}\right)\] 
for all $m\in\mathbb{N}$. Choose some arbitrary $x\in(0,\log(G_k(r)))$. We will construct a sequence in the following manner. First, let $C_0=0$. Now, for each positive integer $l$, let $C_l=C_{l-1}+\log\displaystyle{\left(\sum_{j=0}^{\alpha_l}\frac{1}{p_l^{jr}}\right)}$, where $\alpha_l$ is the largest nonnegative integer less than or equal to $k$ such that $C_{l-1}+\log\displaystyle{\left(\sum_{j=0}^{\alpha_l}\frac{1}{p_l^{jr}}\right)}\leq x$. Also, for each $l\in\mathbb{N}$, let $D_l=\log\displaystyle{\left(\sum_{j=0}^k\frac{1}{p_l^{jr}}\right)-\log\left(\sum_{j=0}^{\alpha_l}\frac{1}{p_l^{jr}}\right)}$, and let $E_l=\displaystyle{\sum_{i=1}^l D_i}$. Note that 
\[\lim_{l\rightarrow\infty}(C_l+E_l)=\lim_{l\rightarrow\infty}\left(\sum_{i=1}^l\log\left(\sum_{j=0}^{\alpha_i}\frac{1}{p_i^{jr}}\right)+\sum_{i=1}^lD_i\right)\]
\[=\lim_{l\rightarrow\infty}\sum_{i=1}^l\log\left(\sum_{j=0}^k\frac{1}{p_i^{jr}}\right)=\log(G_k(r)).\] 
\par 
Now, because the sequence $(C_l)_{l=1}^{\infty}$ is bounded and monotonic, we know that there exists some real number $\gamma$ such that $\displaystyle{\lim_{l\rightarrow\infty}}C_l=\gamma$. Note that, for each $l\in\mathbb{N}$, $C_l$ is in the range of $\log\sigma_{-r,k}$ because 
\[C_l=\sum_{i=1}^l\log\left(\sum_{j=0}^{\alpha_i}\frac{1}{p_i^{jr}}\right)=\log\left(\prod_{i=1}^l\sigma_{-r}(p_i^{\alpha_i})\right)=\log\sigma_{-r,k}\left(\prod_{i=1}^l p_i^{\alpha_i}\right).\] 
Therefore, if we can show that $\gamma=x$, then we will know (because we chose $x$ arbitrarily) that the range of $\log\sigma_{-r,k}$ is dense in $[0,\log(G_k(r)))$, which will complete the proof. 
\par 
Because we defined the sequence $(C_l)_{l=1}^{\infty}$ so that $C_l\leq x$ for all $l\in\mathbb{N}$, we know that $\gamma\leq x$. Now, suppose $\gamma<x$. Then $\displaystyle{\lim_{l\rightarrow\infty}}E_l=\log(G_k(r))-\gamma>\log(G_k(r))-x$. This implies that there exists some positive integer $L$ such that $E_l>\log(G_k(r))-x$ for all $l\geq L$. Let $m$ be the smallest positive integer that satisfies $E_m>\log(G_k(r))-x$. First, suppose $D_m\leq x-C_m$ so that $x\geq C_m+D_m=C_{m-1}+\log\displaystyle{\left(\sum_{j=0}^k\frac{1}{p_m^{jr}}\right)}$. This implies, by the definition of $\alpha_m$, that $\alpha_m=k$. Then $D_m=0$. If $m>1$, then 
$E_{m-1}=E_m>\log(G_k(r))-x$, which contradicts the minimality of $m$. On the other hand, if $m=1$, then we have $0=D_m=E_m>\log(G_k(r))-x$, which is also a contradiction. Thus, we conclude that $D_m>x-C_m$. Furthermore, 
\[\sum_{i=m+1}^{\infty}\log\left(\sum_{j=0}^k\frac{1}{p_i^{jr}}\right)=\log(G_k(r))-\sum_{i=1}^m\log\left(\sum_{j=0}^k\frac{1}{p_i^{jr}}\right)\] 
\begin{equation} \label{Eq2.1}
=\log(G_k(r))-E_m-C_m<x-C_m<D_m,
\end{equation} 
and we originally assumed that $\displaystyle{\log\left(1+\frac{1}{p_m^r}\right)}\leq\sum_{i=m+1}^{\infty}\log\left(\sum_{j=0}^k\frac{1}{p_i^{jr}}\right)$.     
This means that $\displaystyle{\log\left(1+\frac{1}{p_m^r}\right)<D_m=\log\left(\sum_{j=0}^k\frac{1}{p_m^{jr}}\right)-\log\left(\sum_{j=0}^{\alpha_m}\frac{1}{p_m^{jr}}\right)}$, or, \\ equivalently, $\displaystyle{\log\left(1+\frac{1}{p_m^r}\right)+\log\left(\sum_{j=0}^{\alpha_m}\frac{1}{p_m^{jr}}\right)<\log\left(\sum_{j=0}^k\frac{1}{p_m^{jr}}\right)}$. If $\alpha_m>0$, we have 
\[\log\left(\left(1+\frac{1}{p_m^r}\right)^2\right)\leq\log\left(1+\frac{1}{p_m^r}\right)+\log\left(\sum_{j=0}^{\alpha_m}\frac{1}{p_m^{jr}}\right)<\log\left(\sum_{j=0}^k\frac{1}{p_m^{jr}}\right)\] 
\[<\log\left(\sum_{j=0}^{\infty}\frac{1}{p_m^{jr}}\right)=\log\left(\frac{p_m^r}{p_m^r-1}\right),\] 
so $\displaystyle{\left(1+\frac{1}{p_m^r}\right)^2<\frac{p_m^r}{p_m^r-1}}$. We may write this as $\displaystyle{1+\frac{2}{p_m^r}+\frac{1}{p_m^{2r}}<1+\frac{1}
{p_m^r-1}}$, so $\displaystyle{2<\frac{p_m^r}{p_m^r-1}=1+\frac{1}{p_m^r-1}}$. As $p_m^r>2$, this is a contradiction. 
Hence, $\alpha_m=0$. By the definitions of $\alpha_m$ and $C_m$, we see that $\displaystyle{C_{m-1}+\log\left(1+\frac{1}{p_m^r}\right)}$ $>x$ and that $C_m=C_{m-1}$. Therefore, $\displaystyle{\log\left(1+\frac{1}{p_m^r}\right)>x-C_{m-1}=x-C_m}$. \\ 
However, recalling from \eqref{Eq2.1} that $\displaystyle{\sum_{i=m+1}^{\infty}\log\left(\sum_{j=0}^k\frac{1}{p_i^{jr}}\right)<x-C_m}$, we find that $\displaystyle{\sum_{i=m+1}^{\infty}\log\left(\sum_{j=0}^k\frac{1}{p_i^{jr}}\right)<\log\left(1+\frac{1}{p_m^r}\right)}$, which we originally assumed was false. Therefore, $\gamma=x$, so the proof is complete.      
\end{proof}  
Given some positive integer $k$, we may use Theorem \ref{Thm2.1} to find the values of $r>1$ such that the range of $\log\sigma_{-r,k}$ is dense in $[0,\log(G_k(r)))$. To do so, we only need to find the values of $r>1$ such that $f_k(m,r)\leq\log(G_k(r))$ for all $m\in\mathbb{N}$. However, this is still a somewhat difficult problem. Luckily, we can make the problem much simpler with the use of the following theorem. We first need a quick lemma.
\begin{lemma} \label{Lem2.1} 
If $j\in\mathbb{N}\backslash\{1,2,4\}$, then $\displaystyle{\frac{p_{j+1}}{p_j}<\sqrt{2}}$. 
\end{lemma}
\begin{proof}
Pierre Dusart \cite{Dusart10} has shown that, for $x\geq 396\hspace{0.75 mm} 738$, there must be at  least one prime in the interval $\displaystyle{\left[x, x+\frac{x}{25\log^2x}\right]}$. Therefore, whenever $p_j>396\hspace{0.75 mm} 738$, we may set $x=p_j+1$ to get $\displaystyle{p_{j+1}\leq (p_j+1)+\frac{p_j+1}{25\log^2(p_j+1)}}$ $<\sqrt{2}p_j$. Using Mathematica 9.0 \cite{wolfram09}, we may quickly search through all the primes less than $396\hspace{0.75 mm} 738$ to conclude the desired result.   
\end{proof} 
\begin{remark} \label{Rem2.1}
There is an identical statement and proof of Lemma \ref{Lem2.1} in \cite{Defant14}, but we include it again here for the sake of completeness (and so that we may later refer to Lemma \ref{Lem2.1} with a name). 
\end{remark}
\begin{theorem} \label{Thm2.2} 
Let $k\!\in\!\mathbb{N}$, and let $r\!\in\!(1,2]$. The range of the function $\log\sigma_{-r,k}$ is dense in the interval $[0,\log(G_k(r)))$ if and only if \\ $f_k(m,r)\leq\log(G_k(r))$ for all $m\in\{1,2,4\}$. 
\end{theorem} 
\begin{proof} 
In light of Theorem \ref{Thm2.1}, we simply need to show that if \\ 
$f_k(m,r)\leq\log(G_k(r))$ for all $m\in\{1,2,4\}$, then $f_k(m,r)\leq\log(G_k(r))$ for all $m\in\mathbb{N}$. Thus, let us assume that $k$ and $r$ are such that $f_k(m,r)\leq\log(G_k(r))$ for all $m\in\{1,2,4\}$. 
\par 
Now, if $m\in\mathbb{N}\backslash\{1,2,4\}$, then, by Lemma \ref{Lem2.1}, $\displaystyle{\frac{p_{m+1}}{p_m}<\sqrt{2}\leq\sqrt[r]{2}}$, which implies that $\displaystyle{\frac{2}{p_{m+1}^r}>\frac{1}{p_m^r}}$. We then have \[f_k(m+1,r)=\log\left(1+\frac{1}{p_{m+1}^r}\right)+\sum_{i=1}^{m+1}\log\left(\sum_{j=0}^k\frac{1}{p_i^{jr}}\right)\] 
\[\geq 2\log\left(1+\frac{1}{p_{m+1}^r}\right)+\sum_{i=1}^m\log\left(\sum_{j=0}^k\frac{1}{p_i^{jr}}\right)\] 
\[>\log\left(1+\frac{2}{p_{m+1}^r}\right)+\sum_{i=1}^m\log\left(\sum_{j=0}^k\frac{1}{p_i^{jr}}\right)\]
\[>\log\left(1+\frac{1}{p_m^r}\right)+\sum_{i=1}^m\log\left(\sum_{j=0}^k\frac{1}{p_i^{jr}}\right)=f_k(m,r).\]
This means that $f_k(3,r)<f_k(4,r)\leq\log(G_k(r))$. Furthermore, $f_k(m,r)<\log(G_k(r))$ for all $m\geq 5$ because $\displaystyle{(f_k(m,r))_{m=5}^{\infty}}$
is a strictly increasing sequence and $\displaystyle{\lim_{m\rightarrow\infty}f_k(m,r)=\log(G_k(r))}$.  
\end{proof}  
We now have a somewhat simple way to check whether or not the range of $\log\sigma_{-r,k}$ is dense in $[0,\log(G_k(r)))$ for given $k\in\mathbb{N}$ and $r\in(1,2]$, but we can do better. In what follows, we will let $T_k(m,r)=f_k(m,r)-\log(G_k(r))$. 
\begin{lemma} \label{Lem2.2} 
For fixed $k\in\mathbb{N}$ and $m\in\{1,2,4\}$, $T_k(m,r)$ is a strictly increasing function in the variable $r$ for all $\displaystyle{r\in\left(1,\frac{7}{3}\right)}$. 
\end{lemma} 
\begin{proof} 
$\displaystyle{T_k(m,r)=\log\left(1+\frac{1}{p_m^r}\right)-\sum_{i=m+1}^{\infty}\log\left(\sum_{j=0}^k\frac{1}{p_i^{jr}}\right)}$, so, for fixed \\ 
$k\in\mathbb{N}$ and $m\in\{1,2,4\}$, we have \[\frac{d}{dr}T_k(m,r)=\sum_{i=m+1}^{\infty}\left(\left(\frac{\sum_{a=1}^k ap_i^{-ar}}{\sum_{b=0}^k p_i^{-br}}\right)\log p_i\right)-\frac{\log p_m}{p_m^r+1}.\] 
Observe that, for any $p_i\in\mathbb{P}$, $k\in\mathbb{N}$, and $\displaystyle{r\in\left(1,\frac{7}{3}\right)}$, we have \\ 
$\displaystyle{\frac{\sum_{a=1}^k ap_i^{-ar}}{\sum_{b=0}^k p_i^{-br}}\geq\frac{p_i^{-r}}{1+p_i^{-r}}=\frac{1}{p_i^r+1}}$. Therefore, in order to show that \\ 
$\displaystyle{\frac{d}{dr}T_k(m,r)>0}$, it suffices to show that $\displaystyle{\sum_{i=m+1}^{\infty}\frac{\log p_i}{p_i^r+1}}>\frac{\log p_m}{p_m^r+1}$. 
\par 
For each $m\in\{1,2,4\}$, define the function $\displaystyle{J_m\colon\left(1,\frac{7}{3}\right]\rightarrow\mathbb{R}}$ by 
\[J_m(x)=\frac{\log p_m}{p_m^x+1}-\sum_{i=m+1}^{m+6}\frac{\log p_i}{p_i^x+1}.\] 
One may verify, for each $m\in\{1,2,4\}$, that the function $J_m$ is increasing on the interval $\displaystyle{\left(1,\frac{7}{3}\right)}$ and that $\displaystyle{J_m\left(\frac{7}{3}\right)<0}$. Thus, for $m\in\{1,2,4\}$, $\displaystyle{\frac{\log p_m}{p_m^r+1}<\sum_{i=m+1}^{m+6}\frac{\log p_i}{p_i^r+1}<\sum_{i=m+1}^{\infty}\frac{\log p_i}{p_i^r+1}}$. This completes the proof.  
\end{proof} 
\begin{lemma} \label{Lem2.3}
For each positive integer $k$, the functions $T_k(1,r)$ and $T_k(2,r)$ each have precisely one root for $r\in (1,2]$. 
\end{lemma} 
\begin{proof} 
Fix some $k\in\mathbb{N}$. First,  observe that $\displaystyle{\lim_{r\rightarrow 1^+}T_k(1,r)=-\infty}$ and \\ 
$\displaystyle{\lim_{r\rightarrow 1^+}T_k(2,r)=-\infty}$. Also, when viewed as single-variable functions of $r$, $T_k(1,r)$ and $T_k(2,r)$ are continuous over the interval $(1,2]$. Therefore, if we invoke Lemma \ref{Lem2.2} and the Intermediate Value Theorem, we see that it is sufficient to show that $T_k(1,2)$ and $T_k(2,2)$ are positive. We have 
\[T_k(1,2)=\log\left(1+\frac{1}{2^2}\right)-\sum_{i=2}^{\infty}\log\left(\sum_{j=0}^k\frac{1}{p_i^{2j}}\right)>\log\left(\frac{5}{4}\right)-\sum_{i=2}^{\infty}\log\left(\sum_{j=0}^{\infty}\frac{1}{p_i^{2j}}\right)\] 
\[=\log\left(\frac{5}{4}\right)-\log\left(\prod_{i=2}^{\infty}\frac{p_i^2}{p_i^2-1}\right)=\log\left(\frac{5}{4}\right)+\log\left(\frac{4}{3}\right)-\log(\zeta(2))\] 
\[=\log\left(\frac{10}{\pi^2}\right)>0\] 
and 
\[T_k(2,2)=\log\left(1+\frac{1}{3^2}\right)-\sum_{i=3}^{\infty}\log\left(\sum_{j=0}^k\frac{1}{p_i^{2j}}\right)>\log\left(\frac{10}{9}\right)-\sum_{i=3}^{\infty}\log\left(\sum_{j=0}^{\infty}\frac{1}{p_i^{2j}}\right)\] 
\[=\log\left(\frac{10}{9}\right)-\log\left(\prod_{i=3}^{\infty}\frac{p_i^2}{p_i^2-1}\right)=\log\left(\frac{10}{9}\right)+\log\left(\frac{9}{8}\right)+\log\left(\frac{4}{3}\right)-\log(\zeta(2))\] 
\[=\log\left(\frac{10}{\pi^2}\right)>0.\] 
\end{proof} 
\begin{definition} \label{Def2.2}
For $k\in\mathbb{N}$ and $m\in\{1,2,4\}$, we define $R_k(m)$ by 
\[R_k(m)=\begin{cases} r_0, & \mbox{if } T_k(m,r_0)=0 \mbox{ and } 1<r_0<2; \\ 2, & \mbox{if } T_k(m,r)<0 \mbox{ for all } r\in(1,2). \end{cases}\]
Also, for each positive integer $k$, let $M_k$ be the smallest element $m$ of $\{1,2,4\}$ that satisfies $R_k(m)=\min(R_k(1),R_k(2),R_k(4))$. 
\end{definition} 
\begin{remark} \label{Rem2.2} 
Observe that, for each $k\in\mathbb{N}$, Lemma \ref{Lem2.2}, when combined with the fact that $\displaystyle{\lim_{r\rightarrow 1^+}T_k(m,r)=-\infty}$ for all $m\in\{1,2,4\}$, guarantees that the function $R_k$ is well-defined. Furthermore, note that Lemma \ref{Lem2.3} tells us that $R_k(M_k)<2$. Essentially, $M_k$ is the element $m$ of the set $\{1,2,4\}$ that gives $g(r)=T_k(m,r)$ the smallest root in the interval $(1,2)$, and if multiple values of $m$ give $g(r)$ this minimal root, $M_k$ is simply defined to be the smallest such $m$.   
\end{remark} 
\begin{lemma} \label{Lem2.4}
For all $k\in\mathbb{N}$ and $m\in\{1,2,4\}$, $R_{k+1}(m)\geq R_k(m)$, where equality holds if and only if $m=4$ and $R_k(m)=2$. 
\end{lemma} 
\begin{proof} 
Fix $k\in\mathbb{N}$ and $m\in\{1,2,4\}$. Note that if $f_k(m,r)\leq\log(G_k(r))$ for some $r\in(1,2]$, then 
\[f_{k+1}(m,r)-\sum_{i=1}^m\log\left(\sum_{j=0}^{k+1}\frac{1}{p_i^{jr}}\right)=\log\left(1+\frac{1}{p_m^r}\right)\] 
\[=f_{k}(m,r)-\sum_{i=1}^m\log\left(\sum_{j=0}^{k}\frac{1}{p_i^{jr}}\right)\leq\log(G_k(r))-\sum_{i=1}^m\log\left(\sum_{j=0}^{k}\frac{1}{p_i^{jr}}\right)\] 
\[=\sum_{i=m+1}^{\infty}\log\left(\sum_{j=0}^{k}\frac{1}{p_i^{jr}}\right)<\sum_{i=m+1}^{\infty}\log\left(\sum_{j=0}^{k+1}\frac{1}{p_i^{jr}}\right)\] 
\[=\log(G_{k+1}(r))-\sum_{i=1}^m\log\left(\sum_{j=0}^{k+1}\frac{1}{p_i^{jr}}\right),\] 
so $f_{k+1}(m,r)<\log(G_{k+1}(r))$. 
We now consider two cases. 
\vspace{5 mm}
\\ 
Case 1: $T_k(m,r_0)=0$ for some $r_0\!\in\!(1,2)$.  In this case, $R_k(m)=r_0$, so $T_k(m,R_k(m))=0$. Therefore, $f_k(m,R_k(m))=\log(G_k(R_k(m)))$. By the argument made in the preceding paragraph, we conclude that \\ 
$f_{k+1}(m,R_k(m))<\log(G_{k+1}(R_k(m)))$, which is equivalent to the statement $T_{k+1}(m,R_k(m))<0$. Either $R_{k+1}(m)=2>R_k(m)$ or $T_{k+1}(m,R_{k+1}(m))=0>T_{k+1}(m,R_k(m))$. In the latter case, Lemma \ref{Lem2.2} tells us that $R_{k+1}(m)>R_k(m)$.   
\vspace{5 mm}
\\ 
Case 2: $T_k(m,r)<0$ for all $r\in(1,2)$. In this case, $R_k(m)=2$, and $f_k(m,2)\leq\log(G_k(2))$. By the argument made in the beginning of this proof, we conclude that $f_{k+1}(m,2)<\log(G_{k+1}(2))$. Therefore, combining Lemma \ref{Lem2.2} and Definition \ref{Def2.2}, we may conclude that $R_{k+1}(m)=R_k(m)=2$. Note that, by Lemma \ref{Lem2.3}, this case can only occur if $m=4$. 
\end{proof} 
We now mention some numerical results, obtained using Mathematica 9.0, that we will use to prove our final lemma and theorem. 
\par 
Let us define a function $V_k(m,r)$ by \\ 
$\displaystyle{V_k(m,r)=\log\left(1+\frac{1}{p_m^r}\right)-\sum_{i=m+1}^{10^5}\log\left(\sum_{j=0}^k\frac{1}{p_i^{jr}}\right)}$. Then, for fixed $k\in\mathbb{N}$ and $m\in\{1,2,4\}$, we have\[\frac{d}{dr}V_k(m,r)=\sum_{i=m+1}^{10^5}\left(\left(\frac{\sum_{a=1}^k ap_i^{-ar}}{\sum_{b=0}^k p_i^{-br}}\right)\log p_i\right)-\frac{\log p_m}{p_m^r+1}\] 
\[>\sum_{i=m+1}^{m+6}\left(\frac{\log p_i}{p_i^r+1}\right)-\frac{\log p_m}{p_m^r+1}.\] 
Referring to the last two sentences of the proof of Lemma \ref{Lem2.2}, we see that $\displaystyle{\frac{d}{dr}V_k(m,r)>0}$ for $\displaystyle{r\in\left(1,\frac{7}{3}\right)}$ when $k\in\mathbb{N}$ and $m\in\{1,2,4\}$ are fixed. In particular, we will make use of the fact that $V_1(1,r)$ is an increasing function of $r$ on the interval $\displaystyle{\left(1,\frac{7}{3}\right)}$. We may easily verify that $V_1(1,1)<0<V_1\displaystyle{\left(1,\frac{7}{3}\right)}$, so there exists a unique number $\displaystyle{r_1\in\left(1,\frac{7}{3}\right)}$ such that $V_1(1,r_1)=0$. Mathematica approximates this value as $r_1\approx 1.864633$. We have 
\[V_1(1,r_1)=0=T_1(1,R_1(1))=\log\left(1+\frac{1}{2^{R_1(1)}}\right)-\sum_{i=2}^{\infty}\log\left(1+\frac{1}{p_i^{R_1(1)}}\right)\] 
\[<\log\left(1+\frac{1}{2^{R_1(1)}}\right)-\sum_{i=2}^{10^5}\log\left(1+\frac{1}{p_i^{R_1(1)}}\right)=V_1(1,R_1(1)).\]
Because $V_1(1,r)$ is increasing, we find that $R_1(1)>r_1$. The important point here is that $R_1(1)\in(1.8638,2)$. One may confirm, using a simple graphing calculator, that $\displaystyle{\left(1+\frac{1}{2^r}\right)\left(\frac{3^r}{3^r+1}\right)>1+\frac{1}{3^r}}$ for all $r\in(1.8638,2)$. Therefore, we may write
\[T_1(2,R_1(2))=0=T_1(1,R_1(1))=\log\left(1+\frac{1}{2^{R_1(1)}}\right)-\sum_{i=2}^{\infty}\log\left(1+\frac{1}{p_i^{R_1(1)}}\right)\] 
\[=\log\left(\left(1+\frac{1}{2^{R_1(1)}}\right)\left(\frac{3^{R_1(1)}}{3^{R_1(1)}+1}\right)\right)-\sum_{i=3}^{\infty}\log\left(1+\frac{1}{p_i^{R_1(1)}}\right)\] 
\[>\log\left(1+\frac{1}{3^{R_1(1)}}\right)-\sum_{i=3}^{\infty}\log\left(1+\frac{1}{p_i^{R_1(1)}}\right)=T_1(2,R_1(1)).\]
As $T_1(2,r)$ is increasing on the interval $(1,2)$ (by Lemma \ref{Lem2.2}), we find that $R_1(2)>R_1(1)$. We may use a similar argument, invoking the fact that $\displaystyle{\left(1+\frac{1}{2^r}\right)\left(\frac{3^r}{3^r+1}\right)\left(\frac{5^r}{5^r+1}\right)\left(\frac{7^r}{7^r+1}\right)>1+\frac{1}{7^r}}$ for all $r\in(1.8638,2)$, to show that $R_1(4)>R_1(1)$. Thus, $R_1(1)=\min(R_1(1),R_1(2),R_1(4))$, so \\  $M_1=1$. 
\par 
Now, one may easily verify that, for all $r\in(1.67,1.98)$, 
\begin{equation} \label{Eq2.2}
1+\frac{1}{2^r}<\left(1+\frac{1}{3^r}\right)\left(1+\frac{1}{3^r}+\frac{1}{3^{2r}}\right) 
\end{equation}  
and 
\begin{equation} \label{Eq2.3} 
1+\frac{1}{3^r}>\left(\frac{5^r}{5^r-1}\right)\left(\frac{7^r+1}{7^r-1}\right).  
\end{equation} 
If we fix some integer $k\geq 2$, then, for all $r\in(1.67,1.98)$, we may use \eqref{Eq2.2} to write 
\[f_k(1,r)=\log\left(1+\frac{1}{2^r}\right)+\log\left(\sum_{j=0}^k\frac{1}{2^{jr}}\right)\] 
\[<\log\left(\left(1+\frac{1}{3^r}\right)\left(1+\frac{1}{3^r}+\frac{1}{3^{2r}}\right)\right)+\log\left(\sum_{j=0}^k\frac{1}{2^{jr}}\right)\]
\[\leq\log\left(1+\frac{1}{3^r}\right)+\log\left(\sum_{j=0}^k\frac{1}{3^{jr}}\right)+\log\left(\sum_{j=0}^k\frac{1}{2^{jr}}\right)=f_k(2,r).\] 
Similarly, for all $r\in(1.67,1.98)$, we may use \eqref{Eq2.3} to write 
\[f_k(2,r)=\log\left(1+\frac{1}{3^r}\right)+\sum_{i=1}^2\log\left(\sum_{j=0}^k\frac{1}{p_i^{jr}}\right)\] 
\[>\log\left(\left(\frac{5^r}{5^r-1}\right)\left(\frac{7^r+1}{7^r-1}\right)\right)+\sum_{i=1}^2\log\left(\sum_{j=0}^k\frac{1}{p_i^{jr}}\right)\]
\[=\log\left(\sum_{j=0}^{\infty}\frac{1}{5^{jr}}\right)+\log\left(\sum_{j=0}^{\infty}\frac{1}{7^{jr}}\right)+\log\left(1+\frac{1}{7^r}\right)+\sum_{i=1}^2\log\left(\sum_{j=0}^k\frac{1}{p_i^{jr}}\right)\] 
\[>\log\left(\sum_{j=0}^k\frac{1}{5^{jr}}\right)+\log\left(\sum_{j=0}^k\frac{1}{7^{jr}}\right)+\log\left(1+\frac{1}{7^r}\right)+\sum_{i=1}^2\log\left(\sum_{j=0}^k\frac{1}{p_i^{jr}}\right)\]
\[=\log\left(1+\frac{1}{7^r}\right)+\sum_{i=1}^4\log\left(\sum_{j=0}^k\frac{1}{p_i^{jr}}\right)=f_k(4,r).\] 
We now know that $f_k(2,r)>f_k(1,r), f_k(4,r)$ whenever $k\in\mathbb{N}\backslash\{1\}$ and $r\in(1.67,1.98)$. As our last preliminary computation, we need to evaluate $\displaystyle{\lim_{n\rightarrow\infty}R_n(2)}$. For each positive integer $n$, $R_n(2)$ is the unique solution $r\in(1,2)$ of the equation $f_n(2,r)=\log(G_n(r))$. We may rewrite this equation as $\displaystyle{\log\left(1+\frac{1}{3^r}\right)=\sum_{i=3}^{\infty}\log\left(\sum_{j=0}^n\frac{1}{p_i^{jr}}\right)}$, or, equivalently, \\ 
$\displaystyle{\left(\sum_{j=0}^n\frac{1}{2^{jr}}\right)\left(\sum_{j=0}^n\frac{1}{3^{jr}}\right)\left(1+\frac{1}{3^r}\right)}=\prod_{i=1}^{\infty}\left(\sum_{j=0}^n\frac{1}{p_i^{jr}}\right)$. 
Because the summations and the product in this equation converge (for $r>1$) as $n\rightarrow\infty$, we see that $\displaystyle{\lim_{n\rightarrow\infty}R_n(2)}$ is simply the solution (in the interval $(1,2)$) of the equation 
$\displaystyle{\lim_{n\rightarrow\infty}\left[\left(\sum_{j=0}^n\frac{1}{2^{jr}}\right)\left(\sum_{j=0}^n\frac{1}{3^{jr}}
\right)\left(1+\frac{1}{3^r}\right)\right]=\lim_{n\rightarrow\infty}
\left[\prod_{i=1}^{\infty}\left(\sum_{j=0}^n\frac{1}{p_i^{jr}}\right)\right]}$, which we may write as \begin{equation} \label{Eq2.4}
\left(\frac{2^r}{2^r-1}\right)\left(\frac{3^r+1}{3^r-1}\right)=\zeta(r).
\end{equation} 
The only solution to this equation in the interval $(1,2)$ is $r=\eta\approx 1.8877909$ \cite{Defant14}. For now, the important piece of information to note is that \\ 
$\displaystyle{\lim_{n\rightarrow\infty}R_n(2)\in(1.67,1.98)}$.     
\begin{lemma} \label{Lem2.5} 
For all integers $k>1$, $M_k=2$. 
\end{lemma} 
\begin{proof} 
Fix some integer $k>1$. First, suppose $M_k=1$. This means that $R_k(1)\leq R_k(2)$. Using Lemma \ref{Lem2.4} and the facts that $R_1(1)>1.8638$ and $\displaystyle{\lim_{n\rightarrow\infty}R_n(2)<1.98}$, we have 
\[1.8638<R_1(1)<R_k(1)\leq R_k(2)<\lim_{n\rightarrow\infty}R_n(2)<1.98.\]
Therefore, $R_k(1)\in(1.67,1.98)$, so we know that $f_k(2,R_k(1))>f_k(1,R_k(1))$ $=\log(G_k(R_k(1)))$. Hence, $T_k(2,R_k(1))>0$. Lemma \ref{Lem2.2}, when coupled with our assumption that $R_k(1)\leq R_k(2)$, then implies that $T_k(2,R_k(2))>0$. However, this is impossible because Lemma \ref{Lem2.3} and the definition of $R_k(2)$ guarantee that $T_k(2,R_k(2))=0$. 
\par 
Next, suppose $M_k=4$. This means that $R_k(4)<R_k(2)$. Also, referring to Remark \ref{Rem2.2}, we see that $R_k(4)<2$. Therefore, by the definition of $R_k(4)$, we find that $f_k(4,R_k(4))=\log(G_k(R_k(4)))$. Now, we may write 
\[1.8638<R_1(1)<R_1(4)<R_k(4)<R_k(2)<\lim_{n\rightarrow\infty}R_n(2)<1.98.\] 
As $R_k(4)\in(1.67,1.98)$, we have \[f_k(2,R_k(4))>f_k(4,R_k(4))=\log(G_k(R_k(4))).\] Thus, $T_k(2,R_k(4))>0$. Using Lemma \ref{Lem2.2} and our assumption that \\ 
$R_k(4)<R_k(2)$, we get $T_k(2,R_k(2))>0$. Again, this is a contradiction. 
\end{proof} 
We now culminate our work with a final definition and theorem. 
\begin{definition} \label{Def2.3}
Let $\eta_1$ be the unique real number in the interval $(1,2)$ that satisfies \[\left(1+\frac{1}{2^{\eta_1}}\right)^2=\frac{\zeta(\eta_1)}{\zeta(2\eta_1)}.\] For each integer $k>1$, let $\eta_k$ be the unique real number in the interval $(1,2)$ that 
satisfies \[\left(\sum_{j=0}^k\frac{1}{2^{\eta_kj}}\right)\left(\sum_{j=0}^k\frac{1}{3^{\eta_kj}}\right)\left(1+\frac{1}{3^{\eta_k}}\right)=\frac{\zeta(\eta_k)}{\zeta((k+1)\eta_k)}.\]
\end{definition} 
\begin{remark} \label{Rem2.3}
Using Definition \ref{Def2.1} to manipulate the equation \\ $f_k(M_k,R_k(M_k))=\log(G_k(R_k(M_k)))$ and using the fact that 
\[M_k=\begin{cases} 1, & \mbox{if } k=1; \\ 2, & \mbox{if } k\in\mathbb{N}\backslash\{1\}, \end{cases}\]
one can see that $\eta_k$ is simply $R_k(M_k)$. Furthermore, Lemma \ref{Lem2.2} tells us that, for each positive integer $k$, the value of $\eta_k$ is, in fact, unique.
\end{remark}
\begin{theorem} \label{Thm2.3} 
Let $k$ be a positive integer. If $r>1$, then the range of the function $\sigma_{-r,k}$ is dense in the interval $\displaystyle{\left[1,\frac{\zeta(r)}{\zeta((k+1)r)}\right)}$ if and only if $r\leq\eta_k$.
\end{theorem}
\begin{proof}
Let $k$ be a positive integer, and let $\displaystyle{r\in\left(1,\frac{7}{3}\right)}$. Suppose $r\leq\eta_k$. Then, by the definition of $M_k$ and the fact that $\eta_k=R_k(M_k)$, we see that $r\leq R_k(m)$ for all $m\in\{1,2,4\}$. Lemma \ref{Lem2.2} then guarantees that $T_k(m,r)\leq 0$ for all $m\in\{1,2,4\}$, which means that $f_k(m,r)\leq\log(G_k(r))$ for all $m\in\{1,2,4\}$. Theorem \ref{Thm2.2} then tells us that the range of $\log\sigma_{-r,k}$ is dense in the interval $[0,\log(G_k(r)))$, which implies that the range of $\sigma_{-r,k}$ is dense in $[1,G_k(r))$. Now, suppose that $r>\eta_k$. Then $T_k(M_k,r)>T_k(M_k,R_k(M_k))=0$, so, $f_k(M_k,r)>\log(G_k(r))$. This means that the range of $\log\sigma_{-r,k}$ is not dense in $[0,\log(G_k(r)))$, which is equivalent to the statement that the range of $\sigma_{-r,k}$ is not dense in $[1,G_k(r))$. 
\par 
We now need to show that, for any $k\in\mathbb{N}$, the range of $\sigma_{-r,k}$ is not dense in $[0,\log(G_k(r)))$ for all $\displaystyle{r>\frac{7}{3}}$.
To do so, it suffices to show that $f_k(1,r)>\log(G_k(r))$ for all $\displaystyle{r>\frac{7}{3}}$, which means that we only need to show that $\displaystyle{\left(1+\frac{1}{2^r}\right)\sum_{j=0}^k\frac{1}{2^{jr}}>G_k(r)}$ for $\displaystyle{r>\frac{7}{3}}$. Now, because $G_k(r)<\zeta(r)$, we see that it suffices to show that $\displaystyle{\left(1+\frac{1}{2^r}\right)^2>\zeta(r)}$ for $\displaystyle{r>\frac{7}{3}}$. One may easily verify that this inequality holds for $\displaystyle{\frac{7}{3}<r\leq 3}$. For $r>3$, we have 
\[\left(1+\frac{1}{2^r}\right)^2>1+\frac{1}{2^r}+\frac{1}{2}\left(\frac{1}{2^{r-1}}\right)>1+\frac{1}{2^r}+\frac{1}{(r-1)2^{r-1}}\] 
\[=1+\frac{1}{2^r}+\int_2^{\infty}\frac{1}{x^r}dx>\zeta(r).\]

\end{proof}

\section{An Open Problem} 
As the author has done for a density problem related to generalizations divisor functions without restricted domains \cite{Defant14}, we pose a question related to the number of ``gaps" in the range of $\sigma_{-r,k}$ for various $k$ and $r$. That is, given positive integers $k$ and $L$, what are the values of $r>1$ such that the closure of the range of $\sigma_{-r,k}$ is a union of exactly $L$ disjoint subintervals of $\displaystyle{\left[1,\frac{\zeta(r)}{\zeta((k+1)r)}\right]}$? 
\section{Acknowledgements} 
Dedicated to T.B.B. 
\par 
The author would like to thank Professor Pete Johnson for inviting him to the 2014 REU Program in Algebra and Discrete Mathematics at Auburn University.

\bigskip
\hrule
\bigskip

\noindent 2010 {\it Mathematics Subject Classification}:  Primary 11B05; Secondary 11A25. 

\noindent \emph{Keywords: } Dense, divisor function, restriction

\end{document}